\documentclass[reqno]{amsart}

\usepackage{latexsym}
\usepackage[centertags]{amsmath}
\usepackage{amsfonts}
\usepackage{amssymb}
\usepackage{amsthm}
\usepackage{newlfont}
\usepackage{graphics}
\usepackage{color}
\newtheorem{thm}{Theorem}
\newtheorem{cor}[thm]{Corollary}
\newtheorem{lem}[thm]{Lemma}
\newtheorem{prop}[thm]{Proposition}

\theoremstyle{mydefinition}

\theoremstyle{myremark}

\allowdisplaybreaks[1]

\def\increasingS{\textrm{is}}
\def\crossingN{\textrm{cr}}

\def\bdelta{\bar{\delta}}
\def\mb{\mathbf}
\def\NN{\mathbb{N}}

\def\CC{\mathbb{C}}
\def\ZZ{\mathbb{Z}}

\def\ct{\mathop{\mathrm{CT}}}

\def\blambda{\bar{\lambda}}
\def\bmu{\bar{\mu}}
\def\bzero{\bar{\mathbf{0}}}

\newcommand{\qfac}[1]{(q)_n}

\def\sss{\mathfrak{S}}
\title[Determinant Formulas Relating to Tableaux of Bounded Height]{Determinant Formulas Relating to\\ Tableaux of Bounded Height}
\author{Guoce Xin}
\address{Center for
Combinatorics, LPMC, Nankai University, Tianjin 300071, P. R. China}

\email{gxin@nankai.edu.cn}
\date{January 31, 2006}

\begin{document}

\maketitle

\begin{abstract}
Chen et al. recently established bijections for $(d+1)$-noncrossing/
nonnesting matchings, oscillating tableaux of bounded height $d$,
and oscillating lattice walks in the $d$-dimensional Weyl chamber.
Stanley asked what is the total number of such tableaux of length
$n$ and of any shape. We find a determinant formula for the
exponential generating function. The same idea applies to prove
Gessel's remarkable determinant formula for permutations with
bounded length of increasing subsequences. We also give short
algebraic derivations for some results of the reflection principle.
\end{abstract}



\section{Introduction}

For a partition $\lambda=(\lambda_1,\dots,\lambda_d)_{\ge}$ of
length (or height) at most $d$, we associate it with a $\blambda
:=\lambda+(d,d-1,\dots,1)$. Then $\blambda$ belongs to the
$d$-dimensional \emph{Weyl chamber} defined by $W^d=\{\,(x_1,\dots,
x_d): x_1
> \dots
> x_d>0, x_i \in \ZZ \, \}$. In particular, we denote by $\bar{\textbf{0}}=(d,d-1,\dots,1)$
the associate of the empty partition $ \varnothing$.
For $\blambda,\bmu \in W^d$, let $b_n(\blambda;\bmu)$ be the number
of \emph{Weyl oscillating lattice walks} of length $n$, from
$\blambda$ to $\bmu$, staying within $W^d$, with steps
\emph{positive} or \emph{negative} unit coordinate vectors.
\begin{thm}[Grabiner-Magyar \cite{Grabiner-Magyar}, Equation
26] \label{t-Grabiner-Magyar} For fixed $\blambda,\bmu \in W^d$, we
have a determinant formula for the exponential generating function:
\begin{equation}
\label{e-starting-ending-walks} g_{\blambda\bmu}(t)=\sum_{n\ge 0}
b_n(\blambda;\bmu) \frac{t^n}{n!}=\det\left(
I_{\bmu_i-\blambda_j}(2t)-I_{\bmu_i+\blambda_j}(2t) \right)_{1\le
i,j \le d},
\end{equation}
where \begin{equation}\label{e-Bessel} I_s(2t)=[z^s]
\exp(t(z+z^{-1}))=\sum_{n\ge 0 }\frac{1}{n!(n+s)!}t^{2n+s}
\end{equation} is the hyperbolic Bessel function of
the first kind of order $s$.
\end{thm}

Chen et al. \cite{chen} recently established bijections showing that
$(d+1)$-noncrossing (nonnesting) matchings and oscillating tableaux
are in bijection with certain Weyl oscillating lattice walks. Then
Stanley asked (by private communication) the following question: How
many Weyl oscillating lattice walks of length $n$ are there if we
start at $\bzero$ but may end anywhere? Our main result answers this
question:

\begin{thm}\label{t-main} The exponential generating function for
the number of oscillating lattice walks in ${W}^d$ starting at
$\bzero=(d,d-1,\dots,1)$, and with no restriction on the end points
is
\begin{align} \label{e-maintheorem}
G(t):=\sum_{n\ge 0} \sum_{\mu\in W^d} b_n(\bzero;\mu)\frac{t^n}{n!}=
\det( J_{i-j}(2t) )_{1\le i, j\le d},\end{align} where $J_s(2t)=
[z^{s}]\ (1+z) \exp ((z+z^{-1})t)=I_s(2t)+I_{s-1}(2t)$.
\end{thm}

Terminologies not presented here will be given in section 2, where
we will explore the connection of oscillating tableaux with the
Brauer algebra and symplectic group, just as that of standard Young
tableaux (SYTs for short) with the symmetric group and general
linear group. We will see that Theorem \ref{t-main} actually gives a
determinant formula for oscillating tableaux of bounded height,
which is an analogy of Gessel's formula for SYTs of bounded height.

Section 3 is for completeness of section 4, but is of some
independent interest. We describe a simple algebraic derivation of
the hook-length formula (Theorem \ref{t-HLF}) and Theorem
\ref{t-Grabiner-Magyar}, as well as some notations. The method is
easily seen to apply to many other results of the reflection
principle. One can see from the proof a reason why using exponential
generating function is preferable in this context.

Section 4 includes the derivation of Theorem \ref{t-main}. Starting
from the Grabiner-Magyar formula, one can obtain a constant term
expression that can be used to do algebraic calculation. The theorem
is then derived in three key steps: we first apply the
Stanton-Stembridge trick (a kind of symmetrization), then a
classical  formula for symmetric functions, and finally reversely
apply the Stanton-Stembridge trick. The same idea applies to prove a
generalized form (Theorem \ref{t-Gen-Gessel}) of Gessel's remarkable
determinant formula \cite{Gessel}:
\begin{thm}[Gessel]\label{t-Gessel}
Let $u_d(n)$ be the number of permutations on $\{1,2,\dots,n\}$ with
longest increasing subsequences of length at most $d$. Then
 \begin{equation} \sum_{n\ge 0} u_d(n) \frac{t^2}{n!^2}
=\det(I_{i-j}(2t) )_{1\le i,j \le d}.
\end{equation}\end{thm}
Our starting point is the well-known hook-length formula.

%
%

\section{Notations, connections, and applications}
In this section, we will introduce many objects and try to explain
their connections with the classical objects for the symmetric
group, in the view of enumeration. Some of the connections are in
\cite[Section 9]{Stanley_Increasing}, whose notations we shall
closely follow. Finally, we will give some applications of Theorem
\ref{t-main}.

\subsection{Notations}
We assume basic knowledge of the symmetric group $\mathfrak{S}_n$
and its representation. See, e.g., \cite[Chapter 7]{EC2}. Now we
introduce some objects.

The \emph{Brauer algebra} $\mathfrak{B}_n$ (depending on a parameter
$x$ which is irrelevant here) is a certain semisimple algebra with
the underlying space the linear span (say over $\CC$) of (complete)
matchings on $[2n]=\{1,2,\dots, 2n\}$. The dimension of
$\mathfrak{B}_n$ is $$\dim
\mathfrak{B}_n=(2n-1)!!=(2n-1)(2n-3)\cdots 3\cdot 1.$$ Its
irreducible representations are indexed by partitions of $n-2r$, for
$0\le r\le \lfloor n/2 \rfloor$. The dimension of the irreducible
representation $\mathfrak{B}^\mu$ is equal to $\tilde{f}^{\mu}_n$,
that we are going to introduce.

An \emph{oscillating tableau} (or \emph{up-down tableau}) of shape
$\mu$ and length $n$ is a sequence
$(\varnothing=\mu^0,\mu^1,\dots,\mu^{n}=\mu)$ of partitions such
that for all $1\le i\le n -1$, the diagram of $\mu^i$ is obtained
from $\mu^{i-1}$ by either adding or removing one square.
 Denote by $\tilde{f}^{\mu}_n$ the number of such
tableaux. It is known that if $\mu$ is a partition of $n-2r$ for
some nonnegative integer $r$, then
$$\tilde{f}^{\mu}_n = \binom{n}{2r}(2r-1)!! f^{\mu}, \qquad \mu \vdash (n-2r), $$
where $f^\mu$ is the number of standard Young tableaux of shape
$\mu$. See, e.g., \cite[Appendix B6]{Barcelo-Ram} for further
information.

Denote by $\mathfrak{M}_n$ the set of matchings on $[2n]$. A
\emph{matching} $M\in \mathfrak{M}_n$ is a partition of $[2n]$ into
$n$ two-blocks, written in the form $\{\{i_1,j_1\},\dots,
\{i_n,j_n\} \}$. We also write $(i_k,j_k)$ for $\{i_k,j_k\}$ if
$i_k<j_k$. We represent $M$ by a diagram obtained by identifying $i$
with $(i,0)$ in the plane for $i=1,\dots, 2n$, and drawing arcs,
called edges, from $i_k$ to $j_k$ above the horizontal $x$-axis for
all $k$. For $d\ge 2$, a $d$-crossing of a matching $M$ is a set of
$d$ arcs $(i_{r_1},j_{r_1}),
(i_{r_2},j_{r_2}),\dots,(i_{r_d},j_{r_d})$ of $M$ such that
$i_{r_1}<i_{r_2}<\cdots<i_{r_d}<j_{r_1}<j_{r_2}<\cdots <j_{r_d}$. A
matching without any $d$-crossing is called a $d$-noncrossing
matching. We omit here the similar definition of $d$-nesting. Figure
\ref{f-matching} shows the diagram corresponding to the matching
$$M=\{\{1,4\},\{2,8\},\{3,10\},\{5,7\},\{6,9\} \}. $$

\begin{figure}[hbt]
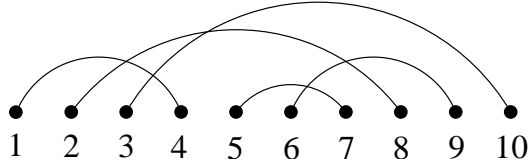

\begin{center}
\input matching.pstex_t
\end{center}
\caption{A matching on [10]\label{f-matching}, in which the edges
$\{1,4\},\{2,8\},\{3,10\}$ form a 3-crossing.}
\end{figure}

\def\pam{\mathrm{bsm}}
Now we introduce apparently new objects. For an oscillating tableau
$O$ of shape $\varnothing$ (hence of even length), reading $O$
backwardly still gives an oscillating tableau of shape
$\varnothing$, denoted by $O^{rev}$. We say that $O$ is
\emph{palindromic} if $O=O^{rev}$. For a matching $M$ of $[2n]$, let
$M^{refl}$ denote the matching obtained from $M$ by reflecting in
the vertical line $x=n+1/2$. Figure \ref{f-matching-rev} shows the
diagram corresponding to $M^{refl}$. Then $M$ is said to be
bilaterally symmetric  if $M=M^{refl}$. Equivalently, $(i,j)$ is an
edge of $M$ if and only if so is $(2n+1-j,2n+1-i)$.
\begin{figure}[hbt]
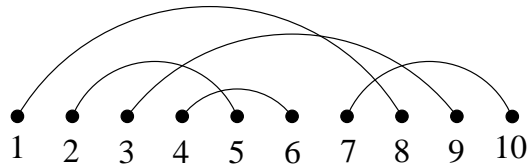

\begin{center}
\input matchingRev.pstex_t
\end{center}
\caption{The reflection of the matching in Figure
\ref{f-matching}.\label{f-matching-rev}}
\end{figure}

\begin{prop}\label{p-palindron}
The exponential generating function of the number $\pam_n$ of
bilaterally symmetric matchings on $[2n]$ is
$$\sum_{n\ge 0} \pam_n\frac{t^n}{n!} =\exp (t+t^2). $$
\end{prop}
\begin{proof}
For a bilaterally symmetric matching $M$ on $[2n]$, identify it with
the graph $M'$ obtained from $M$ by adding the (dashed) edges
$(i,2n+1-i)$ for $i=1,2,\dots,n$. Then every vertex of $M'$ has
degree $2$, so that $M'$ can be uniquely decomposed into connected
components, each being a cycle. The cycles can be of only three
types, as drawn in Figure \ref{f-pam-proof}.
\begin{figure}[hbt]
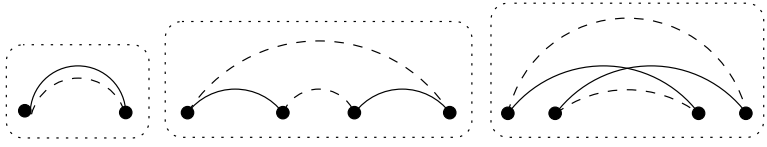

\begin{center}
\input pamProof.pstex_t
\end{center}
\caption{The types of connected components of bilaterally symmetric
matchings.\label{f-pam-proof}}
\end{figure} Therefore, the lemma follows by the well-known exponential formula
for generating functions. See, e.g., \cite[Corollary 5.1.6]{EC2}.
\end{proof}

\subsection{Connections}
We first give a list of the classical objects and their analogies:
$$\begin{array}{l||l}
\textrm{Classical Objects} & \textrm{Their analogies}\\\hline
\textrm{the symmetric group } \mathfrak{S}_n &
\textrm{the Brauer algebra }\mathfrak{B}_n \\
\textrm{the general linear group } GL(d) & \textrm{the symplectic group } Sp(2d)  \\
 \textrm{standard Young tableaux} & \textrm{oscillating tableaux }\\
\textrm{permutations on }[n] & \textrm{matchings on }[2n]  \\
\textrm{involutions} & \textrm{bilaterally symmetric matchings}
\end{array}
$$

Next we give connections in the view of enumeration. By a well-known
result in representation theory, we have
\begin{align}
\label{e-2n-1} \sum_{\mu\vdash (n-2r)} (\tilde{f}^{\mu}_n)^2
=(2n-1)!!, \qquad  \sum_{\mu\vdash n} (f^\mu)^2 =n!,
\end{align}
where the first sum ranges over all nonnegative integers $r$ with
$0\le r \le \lfloor n/2 \rfloor$ and partitions $\mu$ of $n-2r$. The
former equation of \eqref{e-2n-1} is for $\mathfrak{B}_n$ and the
latter one is for $\mathfrak{S}_n$; $\tilde{f}^{\mu}_n$ and $f^\mu$
are the dimension of the corresponding irreducible representations.
We shall always list the analogous formula before the classical one.
Equation \eqref{e-2n-1} suggests a RSK-correspondence for matchings
just as that for permutations. Observe that a pair of oscillating
tableaux of the same shape of length $n$ can be naturally combined
as one oscillating tableau of shape $\varnothing$ of length $2n$. To
be precise, the decomposition $\gamma(O)=(P,Q)$ is given by
$$\gamma: (\varnothing=\mu^0,\mu^1,\dots,\mu^{2n}=\varnothing) \longmapsto ((\varnothing=\mu^0,\mu^1,\dots,\mu^{n}), (\varnothing =\mu^{2n},\mu^{2n-1},\dots,\mu^n)).  $$

Thus it is sufficient to construct a bijection from the set
$\mathfrak{M}_n$ of matchings to the set $\mathcal{O}_n $ of
oscillating tableaux of the empty shape and length $2n$. Such a
bijection was first given by Stanley (unpublished), and was extended
by Sundaram \cite{Sundaram} to arbitrary shapes to give a
combinatorial proof of the Cauchy identity for the symplectic group
$Sp(2d)$, and was recently extended by Chen et al. \cite{chen} for
partitions.

Let $\Phi$ be the bijection from $\mathfrak{M}_n$ to $\mathcal{O}_n$
defined in \cite[Section 9]{Stanley_Increasing}. Then it was shown
in \cite{chen} that $\Phi$ has many properties. We will use the fact
that the maximum number of crossings of a matching $M$ is equal to
the maximum height of the oscillating tableau $\Phi(M)$.

We will use the following result of \cite{Xin-matching-symmetry}:
For any $M\in \mathfrak{M}_n$, we have
$\Phi(M^{refl})=\Phi(M)^{rev}$.

Since the number of palindromic oscillating tableaux is equal to the
number of bilaterally symmetric matchings $\pam_n$, we have, by
Proposition \ref{p-palindron},
\begin{align}
\label{e-2n-2} \sum_{\mu\vdash (n-2r)} \tilde{f}^{\mu}_n
=\left[\frac{t^n}{n!}\right]\exp(t+t^2), \qquad \sum_{\mu\vdash n}
f^\mu =\left[\frac{t^n}{n!}\right]\exp(t+t^2/2),
\end{align}
where the first sum ranges over all nonnegative integers $r$ with
$0\le r \le \lfloor n/2 \rfloor$ and partitions $\mu$ of $n-2r$. The
right equation of \eqref{e-2n-2} counts the number of involutions.

There are analogous results if we put restrictions on the height of
the tableaux. A $d$-oscillating tableau, also called
$d$-\emph{symplectic up-down tableau}, is an oscillating tableau of
a bounded height $d$, by which we mean that the height of every
$\mu^i$ is no larger than $d$. Denote by $\tilde{f}^{\mu}_{n}(d)$
the number of $d$-oscillating tableaux of shape $\mu$ and length
$n$. See \cite{Sundaram} for more information.

There is a natural bijection showing that
$\tilde{f}^{\mu}_{n}(d)=b_n(\bar{\textbf{0}},\bmu)$, which has a
determinant formula as in \eqref{e-starting-ending-walks}. The
bijection simply takes $(\mu^0,\mu^1,\dots,\mu^{n})$ to the sequence
of lattice points $(\bar{\mu}^0,\bmu^1,\dots,\bmu^n)$. Therefore
results on oscillating lattice walks can be translated into those on
oscillating tableaux. Applying $\gamma$ to $d$-oscillating tableaux
of shape $\varnothing$ and length $2n$, and applying Theorem
\ref{t-Grabiner-Magyar}, we obtain
\begin{align}
\label{e-2n-1-d} \sum_{\mu\vdash (n-2r)} (\tilde{f}^{\mu}_n(d))^2
=b_{2n}(\bar{\textbf{0}},\bar{\textbf{0}})=\left[\frac{t^{2n}}{(2n)!}\right]
\det\left( I_{i-j}(2t)-I_{i+j}(2t) \right)_{1\le i,j \le d},
\end{align}
an analogy of Theorem \ref{t-Gessel}. This is also the number of
$(d+1)$-noncrossing/nonnesting matchings. See \cite[Equation
(9)]{chen} (by setting $k=d+1$) and references therein.

Theorem \ref{t-main} actually gives
$$
 \sum_{\mu\vdash (n-2r)} \tilde{f}^{\mu}_n(d) = \left[\frac{t^{n}}{n!}\right]
\det( I_{i-j}(2t)+I_{i-j-1}(2t) )_{1\le i, j\le d}.  $$ This is an
analogy of Gessel's determinant formula for involutions. See
\cite{Gessel}. See also \cite[Sections 4\&5]{Stanley_Increasing}.

Let $\increasingS(w)$ be the length of the longest increasing
subsequences of $w\in \mathfrak{S}_n$, and let $\crossingN(M)$ be
the maximum crossing number of $M \in \mathfrak{M}_n$. Then we have
the following table.

$$\begin{array}{l||l}
\textrm{Classical Objects} & \textrm{Their analogy}\\\hline
\textrm{the general linear group } GL(d) & \textrm{the symplectic group } Sp(2d)  \\
\textrm{SYT of bounded height $d$} & \textrm{oscillating tableaux of bounded height $d$}\\
\{w\in \mathfrak{S}_n: \increasingS(w) \le d \} & \{M\in \mathfrak{M}_n: \crossingN(M)\le d\}   \\
\quad \cdots \quad \& \textrm{ is an involution} & \quad \cdots
\quad \& \textrm{ is bilaterally symmetric}
\end{array}
$$

Stanley \cite{Stanley_Increasing} gave a nice survey for the study
of increasing and decreasing subsequences of permutations and their
variants. One major problem in this area is to understand the
behavior of $\increasingS(w)$. For instance, what is the limiting
distribution of $\increasingS(w)$ for permutations? Gessel's
determinant formula reduces such problem to analysis, which was
solved by Baik, Deift, and Johansson \cite{Baik-Deift-Johansson}
using their techniques. In our table, the limiting distribution
formulas of $\increasingS(w)$ for permutations and for involutions,
and that of $\crossingN(M)$ for matchings are known. The
distribution for bilaterally symmetric matchings should be obtained
in a similar way, but this needs to be checked.

\subsection{Applications}
We first summarize several consequences of Theorem \ref{t-main}.
\begin{cor}\label{c-main}
The following quantities are equal to
$$\left[\frac{t^{n}}{n!}\right]
\det( I_{i-j}(2t)+I_{i-j-1}(2t) )_{1\le i, j\le d}.$$
\begin{enumerate}
\item The number of palindromic Weyl oscillating lattice walks of length $2n$ and starting at
$\bar{\mathbf{0}}$.

\item The number of palindromic oscillating tableaux of length $2n$.

\item The number $\pam_n(d)$ of bilaterally symmetric $(d+1)$-noncrossing/nonnesting matchings on
$[2n]$.

\item The number of oscillating tableaux of any shape and length
$n$.
\end{enumerate}
\end{cor}

We can compute $\pam_n(d)$ for small $d$. For the case $d=1$, we
have
\begin{align}
\pam_{2n}(1)=\binom{2n}{n}, \text{ and }
\pam_{2n+1}(1)=\frac{1}{2}\binom{2n+2}{n+1}. \label{e-consequence-1}
\end{align}

This is a direct consequence of Theorem \ref{t-main}, but we give an
alternative proof.
\begin{proof}[Proof of \eqref{e-consequence-1}]
By Corollary \ref{c-main} part (3), we need to compute noncrossing
bilaterally symmetric matchings on $[2n]$. Let $P(t)$ be the
generating function $P(t) =\sum_{n\ge 0}\pam_{n}(1)t^n$. Consider
the possibility of the edge $(1,m)$ in the bilaterally symmetric
matching $M$. One sees that if $m>n$, then $m$ must equal $2n$ to
avoid a crossing. Thus we have the decomposition of $M$ as in Figure
\ref{f-matching-decomposition}, where we use semi-circles to
indicate noncrossing matchings, and use trapezoid to indicate
noncrossing bilaterally symmetric matchings.
\begin{figure}[hbt]
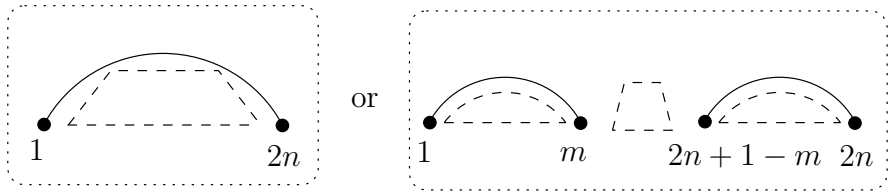

\begin{center}
\input matchingDecomposition.pstex_t
\end{center}
\caption{\label{f-matching-decomposition}Decomposition of
noncrossing bilaterally symmetric matchings.}
\end{figure}
Therefore, we obtain the functional equation:
$$P(t)=1+tP(t)+t^2C(t^2)P(t), $$
where $C(t)=\frac{1-\sqrt{1-4t}}{2t}$ is the Catalan generating
function, which is known to be the ordinary generating function for
noncrossing matchings. Direct algebraic calculation shows
\eqref{e-consequence-1}.
\end{proof}

For the case $d=2$, we have
\begin{align}
\pam_{2n}(2)&
=\frac{1}{2}\binom{2n+2}{n+1}C_n=\frac{(2n+1)!(2n)!}{(n!(n+1)!)^2},\label{e-consequence-2e}\\
\pam_{2n+1}(2)&=\frac{1}{2}\binom{2n+2}{n+1}C_{n+1}=\frac{(2n+1)!(2n+2)!}{n!(n+1)!^2(n+2)!}.\label{e-consequence-2o}
\end{align}

For the case $d=3$, we obtain
\begin{align}
\pam_{2n}(3)&=\sum_{s=0}^n  \frac{2 (2s+1)!}{s!^2(s+1)!(s+2)!} \cdot
\frac{(2n)!}{(n-s)!(n-s+1)!},\label{e-consequence-3e} \\
\pam_{2n+1}(3)&=\sum_{s=0}^n  \frac{2 (2s+2)!}{s!(s+1)!(s+2)!^2}
\cdot \frac{(2n+1)!}{(n-s)!(n-s+1)!}.\label{e-consequence-3o}
\end{align}
By a general theory, these sequences are P-recursive, or their
generating functions (for any $d$) are D-finite. See, e.g.,
\cite[Chapter 6]{EC2}. We use the creative telescoping of \cite{AB}
to find that $\pam_{2n}(3)$ satisfies a second order P-recursion
\cite[Chapter 6]{EC2}:
\begin{multline}
 \left( n+5 \right)  \left( n+4 \right)  \left( n+3 \right)
 \pam_{2n+4}(3)\\=4\left( 5{n}^{2}+30n+43 \right) \left( 2n+3 \right)
\pam_{2n+2}(3)-36\left( 2n+3 \right) \left( 2n+1 \right) \left( n+1
\right) \pam_{2n}(3),
\end{multline}
subject to $\pam_0(3)=1,\pam_2(3)=3$; and that $\pam_{2n+1}(3)$
satisfies a $P$-recursion of order $3$, which is too lengthy to be
given here.

Formulas (\ref{e-consequence-2e},\ref{e-consequence-2o}) are
straightforward by, e.g., the creative telescoping. Formulas
(\ref{e-consequence-3e},\ref{e-consequence-3o}) need some work. We
will write $I_i$ for $I_i(2t)$ for short, and use the following
facts: $I_i=I_{-i}$, $I_{2i}$ contains only even powers in $t$, and
$I_{2i+1}$ contains only odd powers in $t$.

\begin{proof}
[Proof of (\ref{e-consequence-3e},\ref{e-consequence-3o})]
 By Theorem \ref{t-main}, the exponential generating function is
\begin{multline*}
\det \left(\begin {array}{ccc}
I_{{1}}+I_{{0}}&I_{{2}}+I_{{1}}&I_{{2}}+I_
{{3}}\\I_{{1}}+I_{{0}}&I_{{1}}+I_{{0}}&I_{{2}}+I_{{1
}}\\I_{{2}}+I_{{1}}&I_{{1}}+I_{{0}}&I_{{1}}+I_{{0}}
\end {array} \right)\\
=(I_0-I_2)(({I_{{0}}}^{2}-{I_{{1}}
}^{2}-{I_{{2}}}^{2}+I_{{1}}I_{{3}}
)+(I_{{0}}I_{{1}}+I_{{0}}I_{{3}}-2I_{{1}}I_{{2}} )),
\end{multline*}

where $I_0-I_2$ contains only even powers in $t$, and in the right
factor, we have separated the sum according to the parity of the
powers in $t$.

Now it is straightforward, by the creative telescoping, to show that
\begin{align*}
I_0-I_2&=\sum_{n\ge 0} \frac{t^{2n}}{n!(n+1)!},\\
{I_{{0}}}^{2}-{I_{{1}}
}^{2}-{I_{{2}}}^{2}+I_{{1}}I_{{3}}&=\sum_{n\ge 0}
\frac{2(2n+1)!t^{2n}}{n!^2(n+1)!(n+2)!},\\
I_{{0}}I_{{1}}+I_{{0}}I_{{3}}-2I_{{1}}I_{{2}}&=\sum_{n\ge 0} \frac{2
(2n+2)!t^{2n+1}}{n!(n+1)!(n+2)!^2}.
\end{align*}
Equations (\ref{e-consequence-3e},\ref{e-consequence-3o}) then
follow.
\end{proof}

We remark that $\{\pam_{2n}(2)\}_{n\ge 0}$ gives the sequence
A000891 in the Online Encyclopedia of Integer Sequences \cite{OEIS}.
One of its interpretation can be stated in our term: it counts the
number of noncrossing partitions of $[2n+1]$ into $n+1$ blocks. We
also remark that $\{\pam_{2n}(3)\}_{n\ge 0}$ gives the sequence
A064037 in \cite{OEIS}. The only known interpretation is: it counts
the number of 3-dimensional oscillating lattice walks of length
$2n$, starting and ending at the origin, and staying within the
nonnegative octant. Bijections for these objects are desirable.

\section{Algebraic Description of the Reflection Principle}

A classical application of the reflection principle is to the ballot
problem, which, in random walks version, asks how many ways there
are to walk from the origin to a point
$(\lambda_1,\dots,\lambda_d)_{\ge}$, with each step a positive unit
coordinate vector and confined in the region $x_1\ge x_2\ge \cdots
\ge x_d\ge 0$. The reflection principle of
\cite{gessel-zeil,Zeilberger-reflection} gives a determinant
formula, from which the hook-length formula for SYTs can be deduced.
Our objective in this section is to give short algebraic derivations
of this formula and the formula of Grabiner and Magyar.

In the context of lattice walks, it is convenient to shift the
coordinates a little and denote by $W^d=\{\;(x_1,\dots, x_d): x_1
> \dots
> x_d>0, x_i \in \ZZ \; \}$ the $d$-dimensional
\emph{Weyl chamber}. From now on, $\lambda,\mu$ will not denote
partitions as in previous sections. Let $\bdelta=(d,d-1,\dots,1)$.
Then any $\mu\in W^d$ corresponds to a unique partition
$\mu-\bdelta$.

\subsection{Hook-Length Formula}

\begin{thm}[Hook-Length Formula]
\label{t-HLF} The number of standard Young tableaux of shape
$\lambda$ is
\begin{equation}
f^\lambda = (\lambda_1+\cdots+\lambda_d)! \det
\left(\frac{1}{(\lambda_i-i+j)!}\right)_{1\le i,j \le d}.
\end{equation}
\end{thm}

Fixing a starting point $\lambda\in W^d$, we let $f(\lambda; \mu)$
be the number of $W^d$-walks from $\lambda$ to $\mu$, with only
positive unit coordinate vector steps. Clearly, the length of such
walks, if exist, is $|\mu|-|\lambda|$, where
$|\mu|=\mu_1+\cdots+\mu_d$. Then the number $f^{\mu-\bdelta}$ of
SYTs of shape $\mu-\bdelta$ equals $f(\bdelta,\mu)$.

Let $F(x)=F(x_1,\dots,x_d)$ be the generating function
$$F(x_1,\dots,x_d)=\sum_{\mu\in W^d} f(\lambda;\mu) x^\mu , $$
where $x^\mu=x_1^{\mu_1}x_2^{\mu_2}\cdots x_d^{\mu_d}$ records the
end points. From known results, $F(x)$ is $D$-finite and does not
have a simple expression. But $F(x)$ has a simple rational function
\emph{extension}:
\begin{equation}
\bar{F}(x)=\frac{a_\lambda(x)}{1-(x_1+x_2+\cdots+x_d)},
\label{e-extension}
\end{equation}
where $a_\lambda(x)$ is the alternant
$\det\left(x_i^{\lambda_j}\right)_{1\le i,j\le d}$.
\begin{prop}\label{p-Fx}
Let $\bar{F}(x)$ be as above. If we expand
$$\bar{F}(x)=\frac{a_\lambda(x)}{1-(x_1+x_2+\cdots+x_d)}=\sum_{\eta\in \NN^d} \bar{f}(\lambda;\eta) x^\eta , $$
then
$\bar{f}(\lambda;\mu)=f(\lambda;\mu)$ for all $\mu$ in the closure
of $W^d$.
\end{prop}
\begin{proof}
 Let $e_i$ be the $i$th unit coordinate vector. Let $\chi(S)=1$ if the statement $S$ is true and $0$ otherwise.
Then for $\mu$ in the closure of $W^d$, $f(\lambda;\mu)$ can be
uniquely characterized by the following recursion:
\begin{enumerate}
\item[(i)] If $|\mu|\le |\lambda|$, then $f(\lambda;\mu)= \chi(\mu=\lambda).$

\item[(ii)] If $\mu_i=\mu_{i+1}$ for $1\le i \le
d-1 $, then  $f(\lambda;\mu) =0$.

\item[(iii)] If $|\mu|-|\lambda|>0 $, then $f(\lambda;\mu) =\sum_{i=1}^d f(\lambda;\mu-e_i).$
\end{enumerate}
Therefore, it suffices to show that $\bar{f}(\lambda;\mu)$ also
satisfies the above three conditions.

Condition (iii) is trivial according to \eqref{e-extension};
Condition (ii) follows easily from
$$\bar{F}(x_1,\dots,x_{i-1},x_{i+1},x_i,x_{i+2},\dots,x_d)=-\bar{F}(x_1,\dots,x_d);$$
To show condition (i), we notice that the numerator of $\bar{F}(x)$
is homogeneous of degree $|\lambda|$, and the least degree term in
the series expansion of $(1-x_1-\cdots -x_d)^{-1}$ is $1$. This
implies that if $|\eta|\le |\lambda|$, then $\bar{f}(\lambda;\eta)$
equals $(-1)^{\pi}$ if $\eta = \pi(\lambda)$ for some $\pi\in
\mathfrak{S}_d$ and zero otherwise, where $(-1)^\pi$ is the sign of
$\pi$ and $\pi(\lambda)=(\lambda_{\pi_1},\dots,\lambda_{\pi_d})$.
Condition (i) follows since $\pi(\lambda)\in W^d$ only if $\pi$ is
the identity.

This completes the proof.
\end{proof}

By setting $\lambda=\bdelta$, one can derive the hook-length
formula, Theorem \ref{t-HLF}. This completes our first objective of
this section.

\subsection{Grabiner-Magyar Determinant Formula} The same argument applies to more general situations,
such as with a different set of allowing steps. We give one more
example to illustrate the idea. Note that the underlying idea is the
reflection principle.

Fix a starting point $\lambda\in W^d$ (the most interesting case is
$\lambda=\bdelta$). Let $b_n(\lambda;\mu)$ be the number of {Weyl
oscillating lattice walks} of length $n$ from $\lambda$ to $\mu$.
Note that we changed the notation here. The $\blambda$ is
abbreviated by $\lambda$, and similar for $\bmu$.

\begin{prop}\label{p-Bxt}
For fixed $\lambda\in W^d$, let
\begin{equation}
B_\lambda(x;t)= \frac{\det(x_i^{\lambda_j}-x_i^{-\lambda_j})_{1\le
i,j \le d}}{1-t(x_1+x_1^{-1}+x_2+x_2^{-1}+\cdots +x_d+x_d^{-1})}.
\end{equation}
Then $ [x^\mu t^n]\, B_\lambda(x;t)= b_n(\lambda;\mu)$ for any $\mu$
in the closure of $ W^d$ and $n\in \NN$.
\end{prop}
\begin{proof}
Clearly $b_n(\lambda;\mu)$ is uniquely determined by the following
recursion:
\begin{enumerate} \item[(i)] If $n=0$ then $b_n(\lambda;\mu)=
\chi(\mu=\lambda)$.

\item[(ii)] If $\mu_i=\mu_{i+1}$ for $1\le i\le d-1$, or if $\mu_d=0$, then $b_n(\lambda;\mu)
=0$.

\item[(iii)] If $n\ge 1$, then $b_n(\lambda;\mu) =\sum_{i=1}^d
b_{n-1}(\lambda;\mu-e_i)+b_{n-1}(\lambda;\mu+e_i) $.
\end{enumerate}

Denote by $\bar{b}_n(\lambda;\eta)=[x^\eta t^n]\, B_\lambda(x;t)$.
It suffices to show that $\bar{b}_n(\lambda;\mu)$ satisfies the same
recursion as for $b_n(\lambda;\mu)$ when $\mu$ belongs to the
closure of $W^d$. Condition (i) is straightforward; Condition (ii)
follows from the identities $B_\lambda(x;t)=-\left.
B_\lambda(x;t)\right|_{x_i=x_{i+1},x_{i+1}=x_i}$ for $1\le i\le
d-1$, and $B_\lambda(x;t)=-\left.
B_{\lambda}(x;t)\right|_{x_d=-x_d}$; Condition (iii) follows by
writing
$$B_\lambda(x;t)(1-t(x_1+x_1^{-1}+x_2+x_2^{-1}+\cdots +x_d+x_d^{-1}))=\det(x_i^{\lambda_j}-x_i^{-\lambda_j})_{1\le
i,j \le d} $$ and then equating coefficients.
\end{proof}

Now we are ready to complete our second task of this section.

\begin{proof}[Proof of Theorem \ref{t-Grabiner-Magyar}]
By Proposition \ref{p-Bxt}, it remains to extract the coefficient of
$x^\mu$ in $B_\lambda(x;t)$.

Since the numerator of $B_\lambda(x;t)$ is independent of $t$, it is
easy to obtain the exponential generating function
\begin{align*}
\sum_{n\ge 0} \sum_{\eta\in \ZZ^d} \bar{b}_n(\lambda;\eta)x^\eta
\frac{t^{n}}{n!}&= \det(x_i^{\lambda_j}-x_i^{-\lambda_j})_{1\le i,j
\le d}
\exp(t(x_1+x_1^{-1}
+\cdots +x_d+x_d^{-1}))\\
&=\det\left((x_i^{\lambda_j}-x_i^{-\lambda_j})\exp(t(x_i+x_i^{-1}))\right)_{1\le
i,j \le d}.
\end{align*}
Now taking the coefficients of $x_1^{\mu_1}\cdots x_d^{\mu_d}$
yields
$$\sum_{n\ge 0} b_n(\lambda;\mu)
\frac{t^n}{n!}=\det\left( [x_i^{\mu_i-\lambda_j}]
\exp(t(x_i+x_i^{-1}))-[x_i^{\mu_i+\lambda_j}]\exp(t(x_i+x_i^{-1}))
\right)_{1\le i,j\le d} ,$$ which is equivalent to
\eqref{e-starting-ending-walks}.
\end{proof}

Note that we use exponential generating functions because
$\exp(t(x_1+x_1^{-1}+\cdots+x_d+x_d^{-1}))$ factors nicely enough to
be put inside the determinant.

\section{Two Formulas Relating to Tableaux of Bounded Height}
In this section, we will prove Theorems \ref{t-main} and
\ref{t-Gessel}, where the former is a new result and the latter is
Gessel's remarkable determinant formula. We will first express our
objects as certain constant terms. Then we will play two tricks, the
\emph{Stanton-Stembridge trick} and the reverse of the
Stanton-Stembridge trick, in evaluating such constant terms.

\subsection{Stanton-Stembridge Trick}
Fix a working ring $\mathcal{K}$ that includes the ring
$\CC((x_1,\dots,x_d))$ of formal Laurent series as a subring. For
example, in our applications, the working ring is
$\mathcal{K}=\CC((x_1,\dots,x_d))[[t]]$. A permutation $\pi\in
\mathfrak{S}_d$ acts on elements of $\mathcal{K}$ by permuting the
$x$'s, or more precisely
$$\pi\cdot \sum_{i_1,\dots,i_d \in \ZZ} a_{i_1,\dots,i_d}x_1^{i_1}\cdots x_d^{i_d}=
\sum_{i_1,\dots,i_d \in \ZZ} a_{i_1,\dots,i_d}x_{\pi_1}^{i_1}\cdots
x_{\pi_d}^{i_d}. $$ We say that $\mathcal{K}$ is
$\mathfrak{S}_d$\emph{-invariant} if $\pi\cdot
\mathcal{K}=\mathcal{K}$ for any $\pi \in \mathfrak{S}_d$. For
example, the ring $\mathcal{K}=\CC((x_1,\dots,x_d))[[t]]$ is
$\mathfrak{S}_d$-invariant, but the field of iterated Laurent series
$\CC((x_1))((x_2))$ is not $\mathfrak{S}_2$-invariant (see
\cite{xin-iterate}).

In what follows, we always assume that $\mathcal{K}$ is
$\sss_d$-invariant. One can easily check that this condition holds
in our application.

\begin{lem}[Stanton-Stembridge trick]
For any $H(x_1,\dots,x_d)\in \mathcal{K} $, we have
$$\ct_{x_1,\dots,x_d} H(x_1,\dots,x_d)=\frac{1}{d!} \ct_{x_1,\dots,x_d}\sum_{\pi\in \mathfrak{S}_d} \pi\cdot H(x_1,\dots,x_d), $$
where $\ct_{x_1,\dots,x_d}$ means to take the constant term in the
$x$'s.
%
\end{lem}
The lemma obviously holds. The following direct consequence is
useful.

\begin{cor}
Suppose that $H, U, V\in \mathcal{K}$ and that
$U(x)=U(x_1,\dots,x_d)$ and $V(x)=V(x_1,\dots,x_d)$ are symmetric
and antisymmetric in the $x$'s, respectively. Then
\begin{align*}
\ct_{x_1,\dots,x_d} H(x_1,\dots,x_d)U(x)& =\frac{1}{d!}
\ct_{x_1,\dots,x_d}U(x)\sum_{\pi\in \mathfrak{S}_d}
\pi\cdot H(x_1,\dots,x_d),\\
\ct_{x_1,\dots,x_d} H(x_1,\dots,x_d)V(x)& =\frac{1}{d!}
\ct_{x_1,\dots,x_d}V(x)\sum_{\pi\in \mathfrak{S}_d}
(-1)^\pi \pi\cdot H(x_1,\dots,x_d).\\
\end{align*}
\end{cor}
We call both the lemma and the corollary the Stanton-Stembridge
trick (SS-trick for short). See, e.g., \cite[p.~9]{Zeilberger-ASM}.

\subsection{Proof of Theorem \ref{t-main}}
First let us write $g_{\bdelta \mu}(t)$ as a constant term using
Theorem \ref{t-Grabiner-Magyar} and the fact that
$I_s(2t)=I_{-s}(2t)$.
\begin{align*}
g_{\bdelta \mu }(t) &=\det \left(\ct_{x_i}
\left[x_i^{\mu_i-\bdelta_j} \exp((x_i+x_i^{-1})t)
-x_i^{\mu_i+\bdelta_j} \exp((x_i+x_i^{-1})t)\right]\right)_{1\le
i,j\le d}
\end{align*} By factoring
out $x_i^{\mu_i}\exp((x_i+x_i^{-1})t)$ from the $i$th row, we obtain
\begin{align*}
g_{\bdelta \mu }(t)&=\ct_{{x}}  \det \left(
x_i^{-\bdelta_j}-x_i^{\bdelta_j} \right)_{1\le i,j\le d}
\prod_{i=1}^d x_i^{\mu_i}\exp((x_i+x_i^{-1})t).
 \end{align*}

Therefore, $G(t)$ can be expressed as a constant term in the $x$'s:
$$G(t)=\sum_{\mu\in W^d} g_{\bdelta \mu}(t)=\ct_{{x}}
\exp\Big(\sum_{i=1}^d(x_i+x_i^{-1})t\Big) \det \left(
x_i^{-\bdelta_j}-x_i^{\bdelta_j} \right)\cdot \sum_{\mu \in W^d}
x^\mu.
$$

Now we can apply the SS-trick to obtain
\begin{align} \label{e-G-step1}
G(t)=\frac{1}{d!}\ct_{ {x}} \exp\Big(\sum_{i=1}^d
\left(x_i+x_i^{-1}\right)t \Big) \det \left(
x_i^{-\bdelta_j}-x_i^{\bdelta_j} \right) \sum_{\mu \in W^d}
\sum_{\pi\in \sss_d}(-1)^\pi \pi\cdot x^\mu,
\end{align}
where
we used the fact that the first factor is symmetric and the second
factor is antisymmetric in the $x$'s.

The determinant is well-known to be equal to
\begin{align}
\label{e-subs-det} \det \left( x_i^{-\bdelta_j}-x_i^{\bdelta_j}
\right)_{1\le i,j\le d} &=\prod_{i=1}^d(1-x_i^2)\prod _{1\le i <
j\le d}(1-x_{{i}}x_{{j}})a_{\bdelta}(x^{-1}),
\end{align}
where $x^{-1}=(x_1^{-1},\dots,x_d^{-1})$, and the alternants are
related to the Schur functions as follows:
$$a_\mu(x):=\sum_{\pi\in \sss_d}(-1)^\pi \pi\cdot x_1^{\mu_1}\cdots x_d^{\mu_d}=a_{\bdelta}(x) s_{\mu-\bdelta}(x). $$
By the above formula, and the classical identity \cite[Equation
7.52]{EC2} for symmetric functions (by setting $x_{k}=0$ for $k>d$),
we obtain
\begin{align}\label{e-subs-sum}
\sum_{\mu \in W^d} \sum_{\pi\in \sss_d}(-1)^\pi \pi\cdot
x_1^{\mu_1}\cdots x_d^{\mu_d} &= a_{\bdelta}(x)
\frac{1}{\prod_{i=1}^d(1-x_i)\prod_{1\le i<j \le d}(1-x_ix_j) }.
\end{align}

Now substitute \eqref{e-subs-det} and \eqref{e-subs-sum} into
\eqref{e-G-step1}. After a lot of cancelations, we obtain:
\begin{align} \label{e-G-step2}
G(t)=\frac{1}{d!}\ct_{{x}} \exp\Big(\sum_{i=1}^d
\left(x_i+x_i^{-1}\right)t\Big)
\prod_{i=1}^d(1+x_i)a_{\bdelta}(x^{-1})\sum_{\pi\in \sss_d} (-1)^\pi
\pi\cdot x^{\bdelta}.
\end{align}

Now $a_{\bdelta}(x^{-1})$ is antisymmetric. Reversely applying the
SS-trick to \eqref{e-G-step2} gives
\begin{align*}
G(t)&=\ct_{{x}} x_1^{d}\cdots x_d^{1}a_{\bdelta}(x^{-1})
\prod_{i=1}^d\exp((x_i+x_i^{-1})t)
\prod_{i=1}^d(1+x_i) \\
&=\ct_{{x}} \det\left(x_i^{j-i} \right)_{1\le i,j\le d}
\prod_{i=1}^d\exp((x_i+x_i^{-1})t)
\prod_{i=1}^d(1+x_i) \\
&=\det \left( \ct_{x_i} x_i ^{j-i} (1+x_i)\exp((x_i+x_i^{-1})t)
\right)_{1\le i,j\le d},
\end{align*}
which is easily seen to be equivalent to \eqref{e-maintheorem}.

\subsection{Gessel's Determinant Formula}
The tricks for proving Theorem \ref{t-Gessel} are similar as in the
previous subsection. We remark that previous proofs of this result
rely on the powerful tools of symmetric functions. See, e.g.,
\cite{Gessel,Goulden}.

It follows from the RSK-correspondence that
$$u_d(n) =\sum_{|\alpha|=n} f^{\alpha}f^{\alpha}=\sum_{|\mu|=n+|\bdelta|,\mu \in W^d} f(\bdelta;\mu)^2, $$
where $\alpha$ ranges over partitions of $n$ of height at most $d$.
We will find a generating function of
$$u_d(\lambda;n):=\sum_{|\mu|=n+|\lambda|,\mu \in W^d}
f(\lambda;\mu)^2.
$$
More precisely, we have the following generalized form.
\begin{thm}
\label{t-Gen-Gessel} Let $\lambda\in W^d$ and let $I_s(2t)$ be as in
Theorem \ref{t-Grabiner-Magyar}. We have
\begin{equation}
\label{e-U-same} U_d(\lambda;t)=\sum_{n\ge 0} u_d(\lambda;n)
\frac{t^{2n}}{n!^2}=\det(I_{\lambda_i-\lambda_j})_{1\le i,j \le d}.
\end{equation}
\end{thm}
Note that $f(\lambda;\mu)$ is the number of standard skew Young
tableaux of shape $(\mu-\bdelta)/(\lambda-\bdelta)$. See
\cite[Equation 7.7.1]{EC2} (note that there is a change of indices).
Therefore $u_d(\lambda;n)$ counts the number of pairs of standard
skew Young tableaux of the same shape
$(\mu-\bdelta)/(\lambda-\bdelta)$ with $|\mu|-|\lambda|=n$.

\begin{proof}[Proof of Theorem \ref{t-Gen-Gessel}]
By Proposition \ref{p-Fx},
$$f(\lambda;\mu)=[x^\mu]\;  a_\lambda(x) \sum_{k\ge 0} (x_1+x_2+\cdots+x_d)^k . $$
Note that when taking the coefficient in $x^\mu$, only the summand
with respect to $k=n$ has a contribution, where $n=|\mu|-|\lambda|$
is the length of the lattice walks. It follows that
\begin{align*}
f(\lambda;\mu)\frac{t^n}{n!} &=[x^\mu]\;  a_\lambda(x) \exp
\left((x_1+x_2+\cdots+x_d)t\right)\\
&=\det\left( [x_i^{\mu_i-\lambda_j}] \exp(tx_i) \right)_{1\le i,j
\le d}.
\end{align*}
When written in constant term, we obtain
\begin{align}
f(\lambda;\mu)\frac{t^n}{n!} &= \ct_{x} \det(x_i^{\lambda_j})_{1\le
i,j \le d} \prod_{i=1}^d x_i^{-\mu_i}\exp(tx_i)\nonumber \\ &=
\ct_{x} a_\lambda(x^{-1}) x_1^{\mu_1}\cdots x_d^{\mu_d}
\exp(t(x_1^{-1}+\cdots +x_d^{-1})), \label{e-gessel-step1}
\end{align}
where the last equality follows by substituting $x_i^{-1}$ for
$x_i$.

Now squaring both sides of \eqref{e-gessel-step1} and summing over
all $\mu$, we obtain
\begin{align}
U_d(\lambda;t)&= \sum_{\mu\in
W^d}f(\lambda;\mu)^2 \frac{t^{2n}}{n!^2} \nonumber \\
&=\sum_{\mu\in W^d} \ct_{x} a_\lambda(x^{-1}) x^\mu
\exp\Big(t\sum_{i=1}^d x_i^{-1}\Big)\cdot \ct_{y} a_\lambda(y^{-1})
y^{\mu}\exp\Big(t\sum_{i=1}^d y_i^{-1}\Big)\nonumber\\
&=\ct_{x,y} a_\lambda(x^{-1})a_\lambda(y^{-1})
\exp\Big(t\sum_{i=1}^d x_i^{-1}+t\sum_{i=1}^d
y_i^{-1}\Big)\sum_{\mu\in W^d} x^\mu y^\mu.\label{e-gessel-step2}
\end{align}

We need the following easy formula:
\begin{align} \label{e-gessel-alt}
\sum_{\pi,\sigma \in \sss_d} (-1)^\pi (\pi\cdot x^\mu) (-1)^{\sigma}
(\sigma\cdot y^{\mu}) = a_\mu(x) a_\mu(y),
\end{align}
and the well-known Cauchy-Binnet formula (see, e.g.,
\cite[p.~397]{EC2}):
\begin{align}
\label{e-cauchy-binnet} \sum_{\mu\in W^d} a_\mu(x)a_\mu(y) &=
x^{\mb{1}} y^{\mb{1}}\det \left( \frac{1}{1-x_iy_j} \right)_{1\le
i,j \le d},
\end{align}
where $\mb{1}$ is the vector of $d$ $1$'s and
$x^{\mb{1}}=x_1x_2\cdots x_d$.

Now apply the SS-trick to \eqref{e-gessel-step2} for the
$x$-variables and the $y$-variables separately, and then apply
\eqref{e-gessel-alt}. We obtain
\begin{align*}
U_d(\lambda;t) &= \frac{1}{d!^2} \ct_{x,y}
a_\lambda(x^{-1})a_\lambda(y^{-1}) \exp\Big(t\sum_{i=1}^d
x_i^{-1}+t\sum_{i=1}^d y_i^{-1}\Big)\sum_{\mu\in W^d} a_\mu(x)
a_\mu(y).
\end{align*}
Applying \eqref{e-cauchy-binnet} gives
\begin{align*}
U_d(\lambda;t)&=\frac{1}{d!^2} \ct_{x,y}
a_\lambda(x^{-1})a_\lambda(y^{-1}) \exp\Big(t\sum_{i=1}^d
\left(x_i^{-1}+y_i^{-1}\right)\Big) x^{\mb{1}} y^{\mb{1}} \det
\left( \frac{1}{1-x_iy_j} \right).
\end{align*}
Clearly, the last determinant is antisymmetric in the $x$-variables
and also in the $y$-variables. Reversely applying the SS-trick for
the $x$'s and for the $y$'s, we obtain
\begin{align}
U_d(\lambda;t)&=\ct_{x,y} x^{-\lambda} y^{-\lambda}
\exp\Big(t\sum_{i=1}^d \left(x_i^{-1}+y_i^{-1}\right)\Big)
x^{\mb{1}} y^{\mb{1}} \det \left(
\frac{1}{1-x_iy_j} \right)_{1\le i,j \le d}\nonumber\\
&=\ct_{x,y} \det \left(x_i^{1-\lambda_i}y_j^{1-\lambda_j}
\exp(tx_i^{-1}+ty_j^{-1}) \frac{1}{1-x_iy_j} \right)_{1\le i,j \le
d}\nonumber\\
&=\det \left(\ct_{x_i,y_j}x_i^{1-\lambda_i}y_j^{1-\lambda_j}
\exp(tx_i^{-1}+ty_j^{-1}) \frac{1}{1-x_iy_j} \right)_{1\le i,j \le
d}.\label{e-gessel-step3}
\end{align}
We finally need to evaluate the entries (constant terms) of the
above determinant.
\begin{align*}
\ct_{x_i,y_j} \exp(tx_i^{-1}+ty_j^{-1})
\frac{x_i^{1-\lambda_i}y_j^{1-\lambda_j}}{1-x_iy_j} &= \ct_{x_i,y_j}
\sum_{k,l\ge 0}
\frac{t^{k+l}}{k!\; l!}x_i^{-k}y_j^{-l}\sum_{m\ge 0}x_i^{1-\lambda_i+m}y_j^{1-\lambda_j+m}\\
&=\sum_{m\ge \lambda_i-1,\lambda_j-1}
\frac{t^{2-\lambda_i-\lambda_j+2m}}{(1-\lambda_i+m)!(1-\lambda_j+m)!}.
\end{align*}
By changing the indices $1-\lambda_i+m=n$, we obtain
\begin{align}
\ct_{x_i,y_j} \exp(tx_i^{-1}+ty_j^{-1})
\frac{x_i^{1-\lambda_i}y_j^{1-\lambda_j}}{1-x_iy_j} &=
I_{\lambda_i-\lambda_j}(2t).
\end{align}

This completes the proof.
\end{proof}

By going over the proof, one can see that the $a_\lambda(y)$ may be
replaced with $a_\nu(y)$ without making much difference. This gives
the following proposition.

Let $\lambda, \nu \in W^d$ with $|\lambda|\ge |\nu|$. Denote by
$$U_d(\lambda;\nu;t):=\sum_{n\ge 0} u_d(\lambda;\nu;n)\frac{t^{2n+|\lambda|-|\nu|}}{n! (n+|\lambda|-|\nu|)! }, $$
where
$$u_d(\lambda;\nu;n):=\sum_{|\mu|=n+|\lambda|,\mu \in W^d}
f(\lambda;\mu)f(\nu;\mu).
$$
\begin{prop}\label{p-Gessel}
For $\lambda, \nu \in W^d$ with $|\lambda|\ge |\nu|$, we have
$$U_d(\lambda;\nu;t)=\det(I_{\lambda_i-\nu_j})_{1\le i,j\le d}. $$
\end{prop}

Gessel's determinant formula was proved by first deriving a
symmetric function identity, and then applying a specialization
operator. It is not a surprise that there should be a corresponding
symmetric function identity that specializes to Theorem
\ref{t-Gen-Gessel}, even to Proposition \ref{p-Gessel}. Actually,
such a formula was described by Gessel in the same paper
\cite{Gessel} in the paragraph just before Theorem 16, and was
clearly stated as \cite[Theorem 3.5]{yang}.

Gessel's determinant formula counts permutations of bounded length
of longest increasing subsequences. Does Theorem \ref{t-Gen-Gessel}
count natural objects?

%

\vspace{.2cm} \noindent{\bf Acknowledgments.} The author was
grateful to Richard Stanley for asking the question, which motivated
this paper, and for helpful suggestions. The author would also like
to thank Arthur Yang for going carefully over this paper and making
helpful remarks on an earlier draft. This work was supported by the
973 Project, the PCSIRT project of the Ministry of Education, the
Ministry of Science and Technology and the National Science
Foundation of China.


\begin{thebibliography}{99}

\bibitem{Baik-Deift-Johansson}
J. Baik, P. Deift, and K. Johansson, On the distribution of the
length of the longest increasing subsequence of random permutations,
\emph{J. Amer. Math. Soc.}, \textbf{12} (1999) 1119--1178.

\bibitem{Barcelo-Ram}
H. Barcelo and A. Ram, \emph{Combinatorial representation theory},
in New Perspectives in Algebraic Combinatorics (Berkeley, CA,
1996-1997), MSRI Publ. 38, Cambridge University Press, Cambridge,
1999, pp. 23--90.

\bibitem{chen}
W.~Chen, E.~Deng, R.~Du, R.~Stanley, and C.~Yan.
\newblock {Crossings and nestings of matchings and partitions},
\newblock {\emph{Trans. Amer. Math. Soc.}}, \textbf{359} (4) (2007) 1555--1575.

\bibitem{yang}
W. Chen, C. Krattenthaler, and A.L.B. Yang, The flagged Cauchy
determinant, \emph{Graphs and Combinatorics}, \textbf{21} (2005)
51--62.

\bibitem{Gessel}
I.M. Gessel, {Symmetric functions and $P$-recursiveness}, \emph{J.
Combin. Theory Ser. A}, \textbf{53} (1990) 257--285.

\bibitem{gessel-zeil}
I.M. Gessel and D.~Zeilberger,
\newblock {Random walk in a {W}eyl chamber},
\newblock {\emph{Proc. Amer. Math. Soc.}}, \textbf{115} (1992) 27--31.

\bibitem{Goulden}
I.P. Goulden, {A linear operator for symmetric functions and
tableaux in a strip with given trace}, \emph{Discrete Math.},
\textbf{99} (1992) 69--77.

\bibitem{Grabiner-Magyar}
D. J. Grabiner and P. Magyar, {Random walks in Weyl chambers and the
decomposition of the tensor powers}, \emph{J. Algebraic Combin.},
\textbf{2} (1993) 239--260.

\bibitem{AB}
M.~Petkov{\v{s}}ek, H.~S. Wilf, and D.~Zeilberger,
\newblock {\em {$A=B$}},
\newblock A K Peters Ltd., Wellesley, MA, 1996.

\bibitem{OEIS} N. J. A. Sloane, The On-Line Encyclopedia of Integer
Sequences, published electronically at www.research.att.com/\kern
-2pt\lower .8 ex\hbox{\char`\~}njas/sequences/, 2007.



\bibitem{EC2}
R.P. Stanley,
\newblock {\em Enumerative Combinatorics $2$}, vol. 62 of { Cambridge
  Studies in Advanced Mathematics},
\newblock Cambridge University Press, Cambridge, 1999.

\bibitem{Stanley_Increasing}
\bysame, {Increasing and decreasing subsequences of permutations and
their variants}, preprint, arXiv: math.co/0512035.

\bibitem{Sundaram}
S. Sundaram, {The Cauchy identity for $Sp(2n)$}, \emph{J. Comb.
Theory Ser. A}, \textbf{53} (1990) 209--238.




\bibitem{xin-iterate}
G. Xin, {A fast algorithm for {M}ac{M}ahon's partition analysis},
  \textrm{Electron. J. Combin.}, \textbf{11} (2004) R58.

\bibitem{Xin-matching-symmetry}
\bysame, {Symmetry property of Chen et al.'s bijection for set
partitions}, in preparation.

\bibitem{Zeilberger-reflection}
D. Zeilberger, {Andr\'e's reflection proof generalized to the
many-candidate reflection problem}, \emph{Discrete Math.},
\textbf{44} (1983) 325--326.

\bibitem{Zeilberger-ASM}
\bysame, {Proof of the alternating sign matrix conjecture},
\emph{Electron. J. Combin.}, \textbf{3} (1996), R13.

\end{thebibliography}
\end{document}